\documentclass[A4paper]{amsart}
\usepackage[T1]{fontenc}
\usepackage{amsmath,amsthm,amssymb,amsfonts,amscd}
\usepackage{bm} % for bold of greek
\usepackage[noadjust]{cite}
\usepackage{diagmac2} %% for diagrams
\usepackage{tikz}
\usetikzlibrary{positioning,calc,arrows,arrows.meta}
\usepackage{color}

\newtheorem{theorem}{Theorem}[section]
\newtheorem{corollary}[theorem]{Corollary}
\newtheorem{lemma}[theorem]{Lemma}
\newtheorem{proposition}[theorem]{Proposition}

\theoremstyle{definition}
\newtheorem{definition}[theorem]{Definition}
\newtheorem{example}[theorem]{Example}

\theoremstyle{remark}
\newtheorem{remark}[theorem]{Remark}

\numberwithin{equation}{section}

\makeatletter

\renewcommand{\p@enumii}{}
\makeatother

\newcommand{\RR}{\mathbb{R}}

\newcommand{\bB}{\mathbf{B}}

\DeclareMathOperator{\diam}{diam}
\def\<#1>{\langle #1 \rangle}

\newcommand{\acr}{\newline\indent}

\title{Ultrametrics and complete multipartite graphs}

\author{Viktoriia Bilet}
\address{Viktoriia Bilet\acr
Institute of Applied Mathematics and Mechanics of NASU\acr
Dobrovolskogo str. 1, Slovyansk 84100, Ukraine}
\email{viktoriiabilet@gmail.com}

\author{Oleksiy Dovgoshey}
\address{Oleksiy Dovgoshey\acr
Institute of Applied Mathematics and Mechanics of NASU\acr
Dobrovolskogo str. 1, Slovyansk 84100, Ukraine}
\email{oleksiy.dovgoshey@gmail.com}

\author{Yuriy Kononov}
\address{Yuriy Kononov\acr
Institute of Applied Mathematics and Mechanics of NASU\acr
Dobrovolskogo str. 1, Slovyansk 84100, Ukraine}
\email{kononov.yuriy.nikitovich@gmail.com}


\begin{document}

\begin{abstract}
We describe the class of graphs for which all metric spaces with diametrical graphs belonging to this class are ultrametric. It is shown that a metric space \((X, d)\) is ultrametric iff the diametrical graph of the metric \(d_{\varepsilon}(x, y) = \max\{d(x, y), \varepsilon\}\) is either empty or complete multipartite for every \(\varepsilon > 0\). A refinement of the last result is obtained for totally bounded spaces. Moreover, using complete multipartite graphs we characterize the compact ultrametrizable topological spaces. The bounded ultrametric spaces, which are weakly similar to unbounded ones, are also characterized via complete multipartite graphs.
\end{abstract}

\maketitle

\section{Introduction}

In what follows we write \(\RR^{+}\) for the set of all nonnegative real numbers.

\begin{definition}\label{d1.1}
A \textit{semimetric} on a set \(X\) is a function \(d\colon X\times X\rightarrow \RR^{+}\) satisfying the following conditions for all \(x\), \(y \in X\):
\begin{enumerate}
\item \label{d1.1:s1} \((d(x,y) = 0) \Leftrightarrow (x=y)\);
\item \label{d1.1:s2} \(d(x,y)=d(y,x)\).
\end{enumerate}
A semimetric space is a pair \((X, d)\) of a set \(X\) and a semimetric \(d\colon X\times X\rightarrow \RR^{+}\). A semimetric \(d\colon X\times X\rightarrow \RR^{+}\) is called a \emph{metric} if the \emph{triangle inequality}
\[
d(x, y)\leq d(x, z) + d(z, y)
\]
holds \(x\), \(y\), \(z \in X\). A metric \(d\colon X\times X\rightarrow \RR^{+}\) is an \emph{ultrametric} on \(X\) if we have
\begin{equation}\label{d1.1:e3}
d(x,y) \leq \max \{d(x,z),d(z,y)\}
\end{equation}
for all \(x\), \(y\), \(z \in X\). Inequality~\eqref{d1.1:e3} is often called the \emph{strong triangle inequality}. 
\end{definition}

In all ultrametric spaces, each triangle is isosceles with the base being no greater than the legs. The converse statement also is valid: ``If \(X\) is a semimetric space and each triangle in \(X\) is isosceles with the base no greater than the legs, then \(X\) is an ultrametric space.''.

The ultrametric spaces are connected with various of investigations in mathematics, physics, linguistics, psychology and computer science. Some properties of ultrametrics have been studied in~\cite{DM2009, DD2010, DP2013SM, GroPAMS1956, Lemin1984RMS39:5, Lemin1984RMS39:1, Lemin1985SMD32:3, Lemin1988, Lemin2003, Qiu2009pNUAA, Qiu2014pNUAA, BS2017, DM2008, DLPS2008TaiA, KS2012, Vau1999TP, Ves1994UMJ, Ibragimov2012, GH1961S, PTAbAppAn2014, Dov2019pNUAA, DP2020pNUAA, DovBBMSSS2020, VauAMM1975, VauTP2003}. The use of trees and tree-like structures gives a natural language for description of ultrametric spaces \cite{Carlsson2010, DLW, Fie, GV2012DAM, HolAMM2001, H04, BH2, Lemin2003, Bestvina2002, DDP2011pNUAA, DP2019PNUAA, DPT2017FPTA, PD2014JMS, DPT2015, Pet2018pNUAA, DP2018pNUAA, DKa2021, Dov2020TaAoG, BS2017, DP2020pNUAA}. 

The purpose of the present paper is to show that complete multipartite graphs also provide an adequate description of ultrametric spaces in many cases.

Let \((X, d)\) be a metric space. An \emph{open ball} with a \emph{radius} \(r > 0\) and a \emph{center} \(c \in X\) is the set 
\[
B_r(c) = \{x \in X \colon d(c, x) < r\}.
\]
Write \(\bB_X = \bB_{X, d}\) for the set of all open balls in \((X, d)\).

We define the \emph{distance set} \(D(X)\) of a metric space \((X,d)\) as the range of the metric \(d\colon X\times X\rightarrow \RR^{+}\),  
\[
D(X) = D(X, d) := \{d(x, y) \colon x, y \in X\}
\]
and write 
\[
\diam X := \sup \{d(x, y) \colon x, y \in X\}.
\]

The next basic for us notion is a graph.

A \textit{simple graph} is a pair \((V,E)\) consisting of a nonempty set \(V\) and a set \(E\) whose elements are unordered pairs \(\{u, v\}\) of different points \(u\), \(v \in V\). For a graph \(G = (V, E)\), the sets \(V=V(G)\) and \(E = E(G)\) are called \textit{the set of vertices} and \textit{the set of edges}, respectively. We say that \(G\) is \emph{empty} if \(E(G) = \varnothing\). A graph \(G\) is \emph{finite} if \(V(G)\) is a finite set, \(|V(G)| < \infty\). A graph \(H\) is, by definition, a \emph{subgraph} of a graph \(G\) if the inclusions \(V(H) \subseteq V(G)\) and \(E(H) \subseteq E(G)\) are valid.

A \emph{path} is a finite nonempty graph \(P\) whose vertices can be numbered so that
\[
V(P) = \{x_0,x_1, \ldots,x_k\},\ k \geqslant 1, \quad \text{and} \quad E(P) = \bigl\{\{x_0, x_1\}, \ldots, \{x_{k-1}, x_k\}\bigr\}.
\]
In this case we say that \(P\) is a path joining \(x_0\) and \(x_k\).

A graph \(G\) is \emph{connected} if for every two distinct \(u\), \(v \in V(G)\) there is a path in \(G\) joining \(u\) and \(v\). 

The \emph{complement} \(\overline{G}\) of a graph \(G\) is the graph with \(V(\overline{G}) = V(G)\) and such that
\[
\bigl(\{x, y\} \in E(\overline{G})\bigr) \Leftrightarrow \bigl(\{x, y\} \notin E(G)\bigr)
\]
for all distinct \(x\), \(y \in V(G)\).

The following notion of complete multipartite graph is well-known when the vertex set of the graph is finite (see, for example, \cite[p.~17]{Die2005}). Below we need this concept for graphs having the vertex sets of arbitrary cardinality.

\begin{definition}\label{d1.2}
Let \(G\) be a graph and let \(k \geqslant 2\) be a cardinal number. The graph \(G\) is \emph{complete \(k\)-partite} if the vertex set \(V(G)\) can be partitioned into \(k\) nonvoind, disjoint subsets, or parts, in such a way that no edge has both ends in the same part and any two vertices in different parts are adjacent.
\end{definition}

We shall say that $G$ is a \emph{complete multipartite graph} if there is a cardinal number \(k\) such that $G$ is complete $k$-partite. It is easy to prove that if \(G\) is complete multipartite, then the non-adjacency is an equivalence relation on \(V(G)\) having at least two distinct equivalence classes \cite[p.~177]{Die2005}.

Our next definition is a modification of Definition~2.1 from~\cite{PD2014JMS}.

\begin{definition}\label{d1.3}
Let $(X,d)$ be a nonempty metric space. Denote by \(G_{X,d}\) a graph such that \(V(G_{X,d}) = X\) and, for \(u\), \(v \in V(G_{X,d})\),
\begin{equation}\label{d1.3:e1}
(\{u,v\}\in E(G_{X,d}))\Leftrightarrow (d(u,v)=\diam X \text{ and } u \neq v).
\end{equation}
We call $G_{X,d}$ the \emph{diametrical graph} of \((X, d)\).
\end{definition}

\begin{example}\label{ex1.4}
Let \(X\) be a set with \(|X| \geqslant 2\) and let \(G\) be a nonempty graph with \(V(G) = X\). If we define a mapping \(d \colon X \times X \to \RR^{+}\) by
\[
d(x, y) = \begin{cases}
0 & \text{if } x = y,\\
2 & \text{if } \{x, y\} \in E(G),\\
1 & \text{if } \{x, y\} \in E(\overline{G}),
\end{cases}
\]
then \(d\) is a metric on \(X\) and the equality \(G_{X, d} = G\) holds.
\end{example}

\begin{example}\label{ex1.5}
If \((X, d)\) is an unbounded metric space or \(|X| = 1\) holds, then the diametrical graph \(G_{X,d}\) is empty, \(E(G_{X,d}) = \varnothing\).
\end{example}

\begin{remark}\label{r1.6}
The use of the name diametrical graph for graphs generated by metric spaces according to Definition~\ref{d1.3} is not generally accepted. For example, in~\cite{AA2008IJoMaMS, Mul1980JoGT, WLRL2019DM}, a graph \(H\) is said to be diametrical if \(H\) is connected and, for every \(u \in V(H)\), there is the unique \(v \in V(H)\) such that
\[
d_H(u, v) \geqslant d_H(x, y)
\]
holds for all distinct \(x\), \(y \in V(H)\), where \(d_H(x, y)\) is the minimum length of the paths connected \(x\) and \(y\) in \(H\). It can be proved that a connected graph \(H = (V, E)\) with \(|V| \geqslant 2\) is diametrical in this sense if and only if the complement \(\overline{G}_{V, d_H}\) of \(G_{V, d_H}\) is complete multipartite and every part of \(\overline{G}_{V, d_H}\) contains exactly two points.
\end{remark}

The paper is organized as follows.

The necessary facts on metrics and ultrametrics are collected in Section~\ref{sec2}. In particular, Proposition~\ref{p2.7} contains a characterization of totally bounded ultrametric spaces which seems to be new.

The main results of the paper are presented in Section~\ref{sec3}. Theorem~\ref{t3.3} completely describes the class of graphs for which every metric space with diametrical graph from this class is ultrametric. In Proposition~\ref{p2.36} it is shown that diametrical graphs of totally bounded ultrametric spaces are complete multipartite with finite number of parts. In Corollary~\ref{c3.5}, using Proposition~\ref{p2.36}, we find a characterization of ultrametrizable compact topological spaces in terms of complete multipartite graphs. A new characterization of ultrametric spaces and totally bounded ultrametric spaces are given in Theorems~\ref{t5.18} and \ref{t5.9}, respectively. In Theorems~\ref{t2.25} and \ref{t2.35} we study interrelations between bounded and unbounded ultrametrics. In particular, in Theorem~\ref{t2.35} it is shown that the diametrical graph of bounded ultrametric space is empty iff this space is weakly similar to an unbounded ultrametric space.

\section{Some facts on metrics and ultrametrics}
\label{sec2}

First of all, we recall a definition of \emph{total boundedness}.

\begin{definition}\label{d2.3}
A metric space \((X, d)\) is totally bounded if, for every \(r > 0\), there is a finite set \(\{B_r(x_1), \ldots, B_r(x_n)\} \subseteq \bB_X\) such that
\[
X \subseteq \bigcup_{i = 1}^{n} B_r(x_i).
\] 
\end{definition}

An important subclass of totally bounded metric spaces is the class of \emph{compact} metric spaces.

\begin{definition}[Borel---Lebesgue property]\label{d2.5}
A metric space \((X, d)\) is compact if every family \(\mathcal{F} \subseteq \bB_X\) satisfying the inclusion
\[
X \subseteq \bigcup_{B \in \mathcal{F}} B
\]
contains a finite subfamily \(\mathcal{F}_0 \subseteq \mathcal{F}\) such that 
\[
X \subseteq \bigcup_{B \in \mathcal{F}_0} B.
\]
\end{definition}

A standard definition of compactness usually formulated as: Every open cover of a topological space has a finite subcover.

The next proposition seems to be a useful characterization of totally bounded ultrametric spaces.

\begin{proposition}\label{p2.7}
Let \((X, d)\) be a nonempty ultrametric space and let \(\bB_{X}^{r_1}\) be a set of all open balls (in \((X, d)\)) having a fixed radius \(r_1 > 0\),
\begin{equation}\label{p2.7:e1}
\bB_{X}^{r_1} = \{B_{r_1}(c) \colon c \in X\}.
\end{equation}
Then the following conditions are equivalent:
\begin{enumerate}
\item \label{p2.7:s1} \(\bB_{X}^{r_1}\) is finite for every \(r_1 > 0\).
\item \label{p2.7:s2} \((X, d)\) is totally bounded.
\end{enumerate}
\end{proposition}

To prove this proposition, we will use the following lemma.

\begin{lemma}[Corollary~4.5 \cite{DS2021a}]\label{l2.4}
Let \((X, d)\) be an ultrametric space. Then the equivalence 
\[
(B_r(x_1) = B_r(x_2)) \Leftrightarrow (B_r(x_1) \cap B_r(x_2) \neq \varnothing)
\]
is valid for every \(r > 0\) and all \(x_1\), \(x_2 \in X\).
\end{lemma}

\begin{proof}[Proof of Proposition~\(\ref{p2.7}\)]
\(\ref{p2.7:s1} \Rightarrow \ref{p2.7:s2}\). The validity of this implication follows directly from Definition~\ref{d2.3}.

\(\ref{p2.7:s2} \Rightarrow \ref{p2.7:s1}\). Suppose \((X, d)\) is totally bounded. Let \(r_1 > 0\) be given. Then there is a finite set \(\{c_1, \ldots, c_n\} \subseteq X\) such that
\begin{equation}\label{p2.7:e2}
X = \bigcup_{i=1}^{n} B_{r_1}(c_i).
\end{equation}
Moreover, without loss of generality, we assume \(B_{r_1}(c_{n_1}) \neq B_{r_1}(c_{n_2})\) for all distinct \(n_1\), \(n_2 \in \{1, \ldots, n\}\). We claim that the equality
\begin{equation}\label{p2.7:e3}
\bB_{X}^{r_1} = \{B_{r_1}(c_1), \ldots, B_{r_1}(c_n)\}
\end{equation}
holds. Indeed, the inclusion
\[
\{B_{r_1}(c_1), \ldots, B_{r_1}(c_n)\} \subseteq \bB_{X}^{r_1}
\]
follows from \(\{c_1, \ldots, c_n\} \subseteq X\). 

To prove the reverse inclusion, consider an arbitrary \(B \in \bB_{X}^{r_1}\). Using \eqref{p2.7:e2} we can find \(i \in \{1, \ldots, n\}\) such that \(B \cap B_{r_1}(c_i) \neq \varnothing\), that implies the equality \(B = B_{r_1}(c_i)\) by Lemma~\ref{l2.4}. Thus, we have 
\[
B \in \{B_{r_1}(c_1), \ldots, B_{r_1}(c_n)\}
\]
for every \(B \in \bB_{X}^{r_1}\). Equality~\eqref{p2.7:e3} follows.
\end{proof}

The following constructive description of the distance sets of totally bounded ultrametric spaces can be found in \cite{DS2021a}.

\begin{proposition}\label{p2.11}
The following statements are equivalent for every \(A \subseteq \RR^{+}\):
\begin{enumerate}
\item \label{p2.11:s1} There is an infinite totally bounded ultrametric space \((X, d)\) such that \(A\) is the distance set of \((X, d)\).
\item \label{p2.11:s2} There is a strictly decreasing sequence \((x_n)_{n \in \mathbb{N}} \subseteq \RR^{+}\) such that 
\[
\lim_{n \to \infty} x_n = 0
\]
holds and the equivalence 
\[
(x \in A) \Leftrightarrow (x = 0 \text{ or } \exists n \in \mathbb{N} \colon x_n = x)
\]
is valid for every \(x \in \RR^{+}\).
\end{enumerate}
\end{proposition}

In the next section of the paper we will also use a concept of weakly similar ultrametric spaces.

\begin{definition}\label{d2.34}
Let \((X, d)\) and \((Y, \rho)\) be nonempty semimetric spaces. A mapping \(\Phi \colon X \to Y\) is a \emph{weak similarity} of \((X, d)\) and \((Y, \rho)\) if \(\Phi\) is bijective and there is a strictly increasing bijection \(\psi \colon D(Y) \to D(X)\) such that the equality
\begin{equation}\label{d2.34:e1}
d(x, y) = \psi\left(\rho\bigl(\Phi(x), \Phi(y)\bigr)\right)
\end{equation}
holds for all \(x\), \(y \in X\).
\end{definition}

If \(\Phi \colon X \to Y\) is a weak similarity and \eqref{d2.34:e1} holds, then we say that \((X, d)\) and \((Y, \rho)\) are \emph{weakly similar}, and \(\psi\) is the \emph{scaling function} of \(\Phi\). 

Some questions connected with the weak similarities and their generalizations were studied in \cite{DovBBMSSS2020, DLAMH2020, Dov2019IEJA, DP2013AMH, BDSa2020}. The weak similarities of finite ultrametric and semimetric spaces were also considered in \cite{Pet2018pNUAA, GFMV2020a}.

The following lemma is a reformulation of Proposition~1.5 from~\cite{DP2013AMH} (see also Proposition~2.2 in~\cite{BDHM2007TA}).

\begin{lemma}\label{l2.6}
Let \((X, d)\) and \((Y, \rho)\) be nonempty weakly similar semimetric spaces. Then \(d\) is an ultrametric on \(X\) if and only if \(\rho\) is an ultrametric on \(Y\).
\end{lemma}

\section{When diametrical graph are complete and multipartite}
\label{sec3}

Let us start from a refinement of Theorems~3.1 and 3.2 from \cite{DDP2011pNUAA}.

\begin{theorem}\label{t2.24}
Let \((X, d)\) be an ultrametric space with \(|X| \geqslant 2\). Then the following statements are equivalent:
\begin{enumerate}
\item\label{t2.24:s1} The diametrical graph \(G_{X,d}\) of \((X, d)\) is nonempty.
\item\label{t2.24:s2} The diametrical graph \(G_{X,d}\) is complete multipartite.
\end{enumerate}
Furthermore, if \(G_{X, d}\) is complete multipartite, then every part of \(G_{X, d}\) is an open ball with a center in \(X\) and the radius \(r = \diam X\) and, conversely, every open ball \(B_r(c)\) with \(r = \diam X\) and \(c \in X\) is a part of \(G_{X, d}\).
\end{theorem}

\begin{proof}
The validity of \(\ref{t2.24:s1} \Leftrightarrow \ref{t2.24:s2}\) follows from Theorems~3.1 and 3.2 of paper~\cite{DDP2011pNUAA}.

Let \(G_{X, d}\) be a complete multipartite graph, let \(X_{1}\) be a part of \(G_{X, d}\) and let \(x_{1}\) be a point of \(X_{1}\). We claim that the equality
\begin{equation}\label{t2.24:e1}
X_{1} = B_r(x_{1})
\end{equation}
holds with \(r = \diam X\). Using Example~\ref{ex1.5}, we see that the double inequality \(0 < \diam X < \infty\) holds. Hence, the open ball \(B_r(x_1)\) is correctly defined. 

Let \(x_{2}\) be a point of the set \(X \setminus X_1\). Since \(G_{X, d}\) is complete multipartite and \(x_{2} \notin X_{1}\), the membership 
\begin{equation}\label{t2.24:e3}
\{x_1, x_2\} \in E(G_{X, d})
\end{equation}
is valid. From \eqref{t2.24:e3} it follows that
\[
d(x_1, x_2) = \diam X = r.
\]
Hence, \(x_2 \in X \setminus B_r(x_1)\). Thus, the inclusion 
\begin{equation}\label{t2.24:e2}
X \setminus X_1 \subseteq X \setminus B_r(x_1)
\end{equation}
holds.

Similarly, we can prove the inclusion \(X \setminus B_r(x_1) \subseteq X \setminus X_1\). The last inclusion and \eqref{t2.24:e2} imply equality~\eqref{t2.24:e1}.

Let us consider now an open ball \(B_r(c)\) with \(r = \diam X\) and arbitrary \(c \in X\). Then there is a part \(X_2\) of \(G_{X, d}\) such that \(c \in X_2\). Arguing as in the proof of equality \eqref{t2.24:e1}, we obtain the equality \(X_2 = B_r(c)\).
\end{proof}

Theorem~\ref{t2.24} remains valid for all metric spaces \((X, d)\) satisfying the condition: ``If \(t \in D(X)\) and \(t \neq \diam X\), then the inequality
\begin{equation}\label{e3.4}
2t < \diam X
\end{equation}
holds.'' As Example~\ref{ex1.4} shows, the last condition is sharp in the sense that inequality~\eqref{e3.4} cannot be replaced by inequality \(2t \leqslant \diam X\).

\begin{example}\label{ex3.2}
Let us consider a ``metric'' space \((X, d)\) for which the distance between some points can be infinite, i.e., \(d \colon X \times X \to \RR^{+} \cup \{\infty\}\), satisfies the triangle inequality and conditions \ref{d1.1:s1}--\ref{d1.1:s2} from Definition~\ref{d1.1} (see, for example, \cite{BBI2001}). If \((X, d)\) is unbounded, then statements \ref{t2.24:s1}--\ref{t2.24:s2} from Theorem~\ref{t2.24} are equivalent and the set of parts of \(G_{X, d}\) coincides with the set of unbounded open balls
\[
B_{\infty}(c) = \{x \in X \colon d(x, c) < \infty\}, \quad c \in X,
\]
whenever \(G_{X, d}\) is complete multipartite.
\end{example}

The following theorem completely describes the structure of graphs \(H\) for which every metric space \((X, d)\) with \(G_{X, d} = H\) is ultrametric (cf. Remark~\ref{r1.6}).

\begin{theorem}\label{t3.3}
Let \(\Gamma = (V, E)\) be a nonempty graph, \(\overline{\Gamma}\) be the complement of \(\Gamma\) and let \(X\) be the set of vertices of \(\Gamma\), \(X = V(\Gamma)\). Then the following conditions are equivalent:
\begin{enumerate}
\item \label{t3.3:s1} The inequality \(|V(H)| \leqslant 2\) holds for every connected subgraph \(H\) of \(\overline{\Gamma}\).
\item \label{t3.3:s2} For every metric space \((X, d)\) the equality \(G_{X, d} = \Gamma\) implies the ultrametricity of \((X, d)\).
\end{enumerate}
\end{theorem}

\begin{proof}
\(\ref{t3.3:s1} \Rightarrow \ref{t3.3:s2}\). Let \(\Gamma\) satisfy condition~\ref{t3.3:s1} and let \((X, d)\) be a metric space such that
\begin{equation}\label{t3.3:e1}
G_{X, d} = \Gamma.
\end{equation}
If \((X, d)\) is not ultrametric, then there are points \(x\), \(y\), \(z \in X\) satisfying the inequality
\begin{equation}\label{t3.3:e2}
d(x, y) > \max \{d(x, z), d(z, y)\}.
\end{equation}
The inequality \(\diam X \geqslant d(x, y)\), \eqref{t3.3:e1} and \eqref{t3.3:e2} imply
\begin{equation}\label{t3.3:e3}
\{x, z\}, \{z, y\} \in E(\overline{\Gamma}).
\end{equation}
Moreover, from \eqref{t3.3:e2} it follows that the points \(x\), \(y\) and \(z\) are pairwise distinct. Hence, \eqref{t3.3:e3} implies that the graph \(H\) with 
\[
V(H) = \{x, y, z\} \quad \text{and} \quad E(H) = \bigl\{\{x, z\}, \{z, y\}\bigr\}
\]
is connected subgraph of \(\overline{\Gamma}\) for which \(|V(H)| > 2\) holds, contrary to \ref{t3.3:s1}. 

\(\ref{t3.3:s2} \Rightarrow \ref{t3.3:s1}\). Let condition~\ref{t3.3:s2} hold. Suppose that there is a connected subgraph \(H\) of the graph \(\overline{\Gamma}\) such that \(|V(H)| \geqslant 3\). Let \(x\), \(y\) and \(z\) be distinct vertices of \(H\). Without loss of generality, we assume 
\[
\bigl\{\{x, z\}, \{z, y\}\bigr\} \subseteq E(H).
\]
The cases \(\{x, y\} \in E(\overline{\Gamma})\) and \(\{x, y\} \in E(\Gamma)\) are possible. Suppose \(\{x, y\} \in E(\Gamma)\) holds. Let \(a\) and \(b\) be two distinct points of the interval \((1, 2)\). Then we define a metric \(d\) on \(X = V(\Gamma)\) as
\begin{equation}\label{t3.3:e4}
d(u, v) = \begin{cases}
0 & \text{if } u = v,\\
2 & \text{if } \{u, v\} \in E(\Gamma),\\
a & \text{if } \{u, v\} = \{x, z\},\\
b & \text{if } \{u, v\} = \{z, y\},\\
\frac{a+b}{2} & \text{otherwise}.
\end{cases}
\end{equation}
From \(a\), \(b \in (1, 2)\), \eqref{t3.3:e4} and \(E(\Gamma) \neq \varnothing\) it follows that \((X, d)\) is a metric space with the diameter equals \(2\) and the diametrical graph equals \(\Gamma\). In addition, \(a\), \(b \in (1, 2)\) and \eqref{t3.3:e4} imply 
\[
2 = d(x, y) > \max\{d(x, z), d(z, y)\} = \max\{a, b\}.
\]
Similarly, if \(\{x, y\} \in E(\overline{\Gamma})\) holds and \(d \colon X \times X \to \RR^{+}\) is defined by~\eqref{t3.3:e4}. Then we have \(G_{X, d} = \Gamma\) as above and, moreover, 
\[
d(x, y) = \frac{a+b}{2}, \quad d(x, z) = a, \quad d(z, y) = b,
\]
where the numbers \(a\), \(b\), \(\frac{a+b}{2}\) are pairwise different. Thus, the triangle \(\{x, y, z\}\) is not isosceles in both possible cases. Hence, \((X, d)\) is not ultrametric and satisfies \(G_{X, d} = \Gamma\), contrary to \ref{t3.3:s2}.
\end{proof}

Example~\ref{ex1.4} and Theorems~\ref{t2.24}, \ref{t3.3} imply the following.

\begin{corollary}\label{c3.4}
Let \(\Gamma\) be a graph with \(|V(\Gamma)| \geqslant 2\) and let \(|V(H)| \leqslant 2\) hold for every connected subgraph \(H\) of the complement \(\overline{\Gamma}\) of \(\Gamma\). Then \(\Gamma\) is complete multipartite.
\end{corollary}

\begin{proposition}\label{p2.36}
Let \((X, d)\) be a totally bounded ultrametric space with \(|X| \geqslant 2\). Then there is an integer \(k \geqslant 2\) such that the diametrical graph \(G_{X, d}\) is complete \(k\)-partite.
\end{proposition}

\begin{proof}
Since \(|X| \geqslant 2\) holds and every totally bounded metric space is bounded, we have \(0 < \diam X < \infty\). It follows from Proposition~\ref{p2.11} that the equality \(\diam X = d(x_1, x_2)\) holds for some \(x_1\), \(x_2 \in X\). Hence, by Theorem~\ref{t2.24}, \(G_{X, d}\) is complete multipartite. Consequently, there is a cardinal number \(k\) such that \(G_{X, d}\) is complete \(k\)-partite. 

Let \(\{X_i \colon i \in I\}\) be the family of all parts of the diametrical graph \(G_{X, d}\) and let \(r_1 := \diam X\). Then, by Theorem~\ref{t2.24}, we have 
\begin{equation}\label{p2.36:e1}
\{X_i \colon i \in I\} \subseteq \bB_{X}^{r_1},
\end{equation}
where \(\bB_{X}^{r_1}\) is the set of all open balls (in \((X, d)\)) with the radius \(r_1\) and \(\operatorname{card} I = k\). By Proposition~\ref{p2.7}, the set \(\bB_{X}^{r_1}\) is finite. Hence, \(k\) is finite by inclusion~\eqref{p2.36:e1}.
\end{proof}

Using Proposition~\ref{p2.36} and Theorem~\ref{t3.3} we obtain the following.

\begin{corollary}\label{c3.6}
Let \((X, d)\) be a totally bounded ultrametric space. If every metric space \((X, \rho)\) satisfying the equality \(G_{X, d} = G_{X, \rho}\) is ultrametric, then \((X, d)\) is finite.
\end{corollary}

To formulate the next corollary, we recall some concepts from General Topology.

\begin{definition}\label{d3.3}
Let \(\tau\) and \(d\) be a topology and, respectively, an ultrametric on a set \(X\). Then \(\tau\) and \(d\) are said to be compatible if \(\bB_{X, d}\) is an open base for the topology \(\tau\).
\end{definition}

Definition~\ref{d3.3} means that \(\tau\) and \(d\) are compatible if and only if every \(B \in \bB_{X, d}\) belongs to \(\tau\) and every \(A \in \tau\) can be written as the union of a family of elements of \(\bB_{X, d}\). If \((X, \tau)\) admits a compatible with \(\tau\) ultrametric on \(X\), then we say that the topological space \((X, \tau)\) is \emph{ultrametrizable}.

\begin{lemma}\label{t3.4}
Let \((X, \tau)\) be an ultrametrizable nonempty topological space. Then the following conditions are equivalent:
\begin{enumerate}
\item\label{t3.4:s1} The space \((X, \tau)\) is compact.
\item\label{t3.4:s2} The distance set \(D(X, d)\) has the largest element whenever \(d\) is a compatible with \(\tau\) ultrametric.
\end{enumerate}
\end{lemma}

This lemma follows directly from Theorem~4.7 of \cite{DS2021a}.

\begin{corollary}\label{c3.5}
Let \((X, \tau)\) be an ultrametrizable topological space with \(\operatorname{card} X \geqslant 2\). Then the following conditions are equivalent:
\begin{enumerate}
\item \label{c3.5:s1} The diametrical graph \(G_{X, d}\) is complete \(k\)-partite with some integer \(k = k(d)\) whenever \(d\) is a compatible with \(\tau\) ultrametric.
\item \label{c3.5:s2} The diametrical graph \(G_{X, d}\) is complete multipartite whenever \(d\) is a compatible with \(\tau\) ultrametric.
\item \label{c3.5:s3} The topological space \((X, \tau)\) is compact.
\end{enumerate}
\end{corollary}

\begin{proof}
\(\ref{c3.5:s1} \Rightarrow \ref{c3.5:s2}\). This implication is evidently valid.

\(\ref{c3.5:s2} \Rightarrow \ref{c3.5:s3}\). Suppose that \ref{c3.5:s2} holds. Let \(d \colon X \times X \to \RR^{+}\) be a compatible with \(\tau\) ultrametric. Then, by Theorem~\ref{t2.24}, there are points \(x_1\), \(x_2 \in X\) such that \(d(x_1, x_2) = \diam X\). Hence, the distance set \(D(X, d)\) contains the largest element. It implies the compactness of \((X, \tau)\) by Lemma~\ref{t3.4}.

\(\ref{c3.5:s3} \Rightarrow \ref{c3.5:s1}\). Since every compact ultrametric space is totally bounded, the validity of \(\ref{c3.5:s3} \Rightarrow \ref{c3.5:s1}\) follows from Proposition~\ref{p2.36}.
\end{proof}

\begin{remark}\label{r3.6}
Necessary and sufficient conditions under which topological spaces are ultrametrizable were found by De Groot \cite{GroPAMS1956, GroCM1958}. See also \cite{CLTaiA2020, BMTA2015, KSBLMS2012, BriTP2015, CSJPAA2019, DS2021a} for future results connected with ultrametrizable topologies.
\end{remark}

\begin{example}\label{ex5.15}
Let \(\overline{B}_1(0) = \{x \in \mathbb{Q}_p \colon d_p(x, 0) \leqslant 1\}\) be the unit closed ball in the ultrametric space \((\mathbb{Q}_p, d_p)\) of \(p\)-adic numbers. Then \(\overline{B}_1(0)\) is a compact infinite subset of \((\mathbb{Q}_p, d_p)\) (Theorem~5.1, \cite{Sch1985}). Hence, by Proposition~\ref{p2.36}, the diametrical graph \(G_{\overline{B}_1(0), d_p|_{\overline{B}_1(0) \times \overline{B}_1(0)}}\) is complete \(k\)-partite with some integer \(k \geqslant 2\). Since the ball \(\overline{B}_1(0)\) can be written as disjoint union of open balls,
\begin{equation}\label{ex5.15:e1}
\overline{B}_1(0) = B_1(0) \cup B_1(1) \cup \ldots \cup B_1(p-1)
\end{equation}
(see, for example, Problem~50 in \cite{Gou1993}), the diametrical graph of \(\overline{B}_1(0)\) is complete \(p\)-partite with the parts \(B_1(i) \in \bB_{\mathbb{Q}_p}\), \(i=0\), \(1\), \(\ldots\), \(p-1\), by Theorem~\ref{t2.24}.
\end{example}

Definition~\ref{d1.3} of diametrical graph can be generalized by following way.

Let \((X, d)\) be a metric space with \(|X| \geqslant 2\) and let \(r \in (0, \infty]\). Denote by \(G_{X, d}^{r}\) a graph such that \(V(G_{X, d}^{r}) = X\) and, for \(u\), \(v \in V(G_{X, d}^{r})\),
\begin{equation}\label{e5.21}
\bigl(\{u, v\} \in E(G_{X, d}^{r})\bigr) \Leftrightarrow \bigl(d(u, v) \geqslant r\bigr).
\end{equation}

\begin{remark}\label{r5.18}
It is clear that \eqref{d1.3:e1} and \eqref{e5.21} are equivalent if \(r = \diam X\). Consequently, we have the equality \(G_{X, d}^{r} = G_{X, d}\) for \(r = \diam X\). In particular, the equality \(G_{X, d}^{\infty} = G_{X, d}\) holds if \((X, d)\) is unbounded.
\end{remark}

Now we can give a new characterization of ultrametric spaces.

\begin{theorem}\label{t5.18}
Let \((X, d)\) be a metric space with \(|X| \geqslant 2\). Then the following statements are equivalent:
\begin{enumerate}
\item \label{t5.18:s1} The metric space \((X, d)\) is ultrametric.
\item \label{t5.18:s2} \(G_{X, d}^{r}\) is either empty or complete multipartite for every \(r \in (0, \diam X]\).
\end{enumerate}
\end{theorem}

\begin{proof}
\(\ref{t5.18:s1} \Rightarrow \ref{t5.18:s2}\). Let \((X, d)\) be ultrametric, let \(r \in (0, \diam X]\) and let a function \(\psi_r \colon \RR^{+} \to \RR^{+}\) be defined as
\begin{equation}\label{t5.18:e1}
\psi_r(t) = \min\{r, t\}, \quad t \in \RR^{+}.
\end{equation}
It is easy to prove that the mapping \(\rho_r = \psi_r \circ d\) is an ultrametric on \(X\). From \eqref{t5.18:e1} and \(r \in (0, \diam X] = (0, \diam(X, d)]\) it follows that \(\diam (X, \rho_r) = r\). The last equality and \eqref{e5.21} imply
\begin{equation}\label{t5.18:e4}
G_{X, \rho_r} = G_{X, d}^{r}.
\end{equation}
By Theorem~\ref{t2.24}, the diametrical graph \(G_{X, \rho_r}\) is either empty or complete multipartite. The validity of \(\ref{t5.18:s1} \Rightarrow \ref{t5.18:s2}\) follows.

\(\ref{t5.18:s2} \Rightarrow \ref{t5.18:s1}\). Let \ref{t5.18:s2} hold. Suppose that there are \(x_1\), \(x_2\), \(x_3 \in X\) satisfying
\begin{equation}\label{t5.18:e2}
d(x_1, x_2) > \max\{d(x_1, x_3), d(x_3, x_2)\}.
\end{equation}
Let us consider \(G_{X, d}^{r}\) with \(r = d(x_1, x_2)\). It is clear that \(G_{X, d}^{r}\) is a nonempty graph. Inequality \eqref{t5.18:e2} implies that the points \(x_1\), \(x_2\), \(x_3\) are pairwise distinct. In the correspondence with \ref{t5.18:s2}, \(G_{X, d}^{r}\) is complete multipartite. Let \(X_i\) be a part of \(G_{X, d}^{r}\) such that \(x_i \in X_i\) holds, \(i = 1\), \(2\), \(3\). By \eqref{e5.21}, we have \(\{x_1, x_2\} \in E(G_{X, d}^{r})\). Hence, \(X_1\) are \(X_2\) are distinct, \(X_1 \neq X_2\). If \(X_1 = X_3\) holds, then from \eqref{e5.21} it follows that
\begin{equation}\label{t5.18:e3}
d(x_2, x_3) \geqslant r = d(x_1, x_2),
\end{equation}
contrary to \eqref{t5.18:e2}. Thus, we have \(X_1 \neq X_3\). Similarly, we obtain \(X_2 \neq X_3\). Hence, \(X_1\), \(X_2\), \(X_3\) are distinct parts of \(G_{X, d}^{r}\). The last statement also implies \eqref{t5.18:e3}, that contradicts \eqref{t5.18:e2}. It is shown that the strong triangle inequality holds for all \(x_1\), \(x_2\), \(x_3 \in X\). The validity of \(\ref{t5.18:s2} \Rightarrow \ref{t5.18:s1}\) follows.
\end{proof}

For the case of totally bounded ultrametric spaces we have the following refinement of Theorem~\ref{t5.18}.

\begin{theorem}\label{t5.9}
Let \((X, d)\) be a metric space with \(|X| \geqslant 2\). Then \((X, d)\) is totally bounded and ultrametric if and only if \(G_{X, d}^{r}\) is complete \(k\)-partite with an integer \(k = k(r)\) for every \(r \in (0, \diam X]\).
\end{theorem}

\begin{proof}
Suppose that \((X, d)\) is totally bounded and ultrametric. Let \(r \in (0, \diam X]\) and let \(\psi_r \colon \RR^{+} \to \RR^{+}\) be defined by \eqref{t5.18:e1}
\begin{equation*}
\psi_r(t) = \min\{r, t\}, \quad t \in \RR^{+}.
\end{equation*}
Then 
\begin{equation}\label{t5.9:e5}
\rho_r = \psi_r \circ d.
\end{equation}
is an ultrametric on \(X\). Moreover, \eqref{t5.9:e5} implies that, for every \(c \in X\), we have
\[
\{x \in X \colon d(x, c) < r_0\} = \{x \in X \colon d_{\rho_r}(x, c) < r_0\}
\]
whenever \(0 < r_0 \leqslant r\) holds. Thus, the ultrametric spaces \((X, d)\) and \((X, \rho_r)\) have the same sets of open balls with a radius at most \(r\). Now using Definition~\ref{d2.3}, we see that \((X, \rho_r)\) is a totally bounded ultrametric space. By Proposition~\ref{p2.36}, the diametrical graph \(G_{X, \rho_r}\) is complete \(k\)-partite for an integer \(k = k(r)\). As in the proof of Theorem~\ref{t5.18}, we obtain the equality
\begin{equation}\label{t5.9:e9}
G_{X, \rho_r} = G_{X, d}^{r}.
\end{equation}
Hence, \(G_{X, d}^{r}\) is also \(k\)-partite with the same \(k\).

Suppose now that, for every \(r \in (0, \diam(X, d)]\), \(G_{X, d}^{r}\) is complete \(k\)-partite with an integer \(k = k(r)\). Using Theorem~\ref{t5.18} we obtain that \((X, d)\) is ultrametric. 

Let \(r \in (0, \diam(X, d)]\) be given. Then the space \((X, \rho_r)\) is also ultrametric. Now equality~\eqref{t5.9:e9} and the second part of Theorem~\ref{t2.24} imply that there are points \(x_1\), \(\ldots\), \(x_{k(r)} \in X\) such that 
\begin{equation}\label{t5.9:e6}
X \subseteq \bigcup_{i=1}^{k(r)} B_{r^{*}}^{\rho}(x_i)
\end{equation}
where
\begin{equation}\label{t5.9:e7}
r^{*} = \diam (X, \rho_r) \quad \text{and} \quad B_{r^{*}}^{\rho}(x_i) = \{x \in X \colon \rho_r(x, x_i) < r^{*}\}.
\end{equation}
From \eqref{t5.9:e5} and the first equality in \eqref{t5.9:e7} it follows that
\begin{equation}\label{t5.9:e8}
B_{r^{*}}^{\rho}(x_i) \subseteq B_{r}(x_i) = \{x \in X \colon d(x, x_i) < r\}
\end{equation}
for every \(i \in \{1, \ldots, k(r)\}\). Since \(r\) is an arbitrary point of \((0, \diam (X, d)]\), Definition~\ref{d2.3} and formulas \eqref{t5.9:e6}, \eqref{t5.9:e8} imply the total boundedness of \((X, d)\).
\end{proof}

The following result is similar to Theorem~2.2 from~\cite{BDSa2020} whose proof is based on properties of ultrametric preserving functions. It is interesting to note that the concept of semimetric spaces allows us not to use the ultrametric preserving functions in the proof below.

\begin{theorem}\label{t2.25}
Let \((X, d)\) be an unbounded ultrametric space, let \(d^{*} \in (0, \infty)\) and \(\rho \colon X \times X \to \RR^{+}\) be defined as
\begin{equation}\label{t2.25:e1}
\rho(x, y) = \frac{d^{*} \cdot d(x, y)}{1 + d(x, y)}.
\end{equation}
Then \((X, \rho)\) is a bounded ultrametric space with empty diametrical graph \(G_{X, \rho}\).

Conversely, let \((X, \rho)\) be a bounded ultrametric space with \(|X| \geqslant 2\) and empty \(G_{X, \rho}\). Write \(d^{*} = \diam (X, \rho)\). Then there is an unbounded ultrametric space \((X, d)\) such that~\eqref{t2.25:e1} holds for all \(x\), \(y \in X\).
\end{theorem}

\begin{proof}
It is clear that the mapping \(\rho \colon X \times X \to \RR^{+}\), defined by~\eqref{t2.25:e1}, is a semimetric. Let us define a function \(f \colon \RR^{+} \to \RR^{+}\) as
\begin{equation}\label{t2.25:e2}
f(t) = \frac{d^{*} t}{1+t}
\end{equation}
for all \(t \in \RR^{+}\). Since \(f\) is strictly increasing and satisfies the equality \(f(0) = 0\), the identical mapping \(\operatorname{Id} \colon X \to X\) is a weak similarity of \((X, d)\) and \((X, \rho)\). By Lemma~\ref{l2.6}, the semimetric \(\rho \colon X \times X \to \RR^{+}\) is an ultrametric. Now from
\begin{equation*}
\lim_{t \to \infty} f(t) = d^{*},
\end{equation*}
we obtain
\[
\rho(x, y) < \lim_{t \to \infty} f(t) = d^{*} = \diam (X, \rho)
\]
for all \(x\), \(y \in X\). Thus, the diametrical graph \(G_{X, \rho}\) is empty.

Conversely, let \((X, \rho)\) be a bonded ultrametric space with \(|X| \geqslant 2\) and empty diametrical graph \(G_{X, \rho}\). Write \(d^{*} = \diam (X, \rho)\). The inequality \(|X| \geqslant 2\) and boundedness of \((X, \rho)\) imply \(d^{*} \in (0, \infty)\). The function \(g \colon [0, d^{*}) \to \RR^{+}\),
\begin{equation}\label{t2.25:e7}
g(s) = \frac{s}{d^{*} - s},
\end{equation}
is strictly increasing and satisfies the equalities
\begin{equation}\label{t2.25:e5}
g(0) = 0 \quad \text{and} \quad \lim_{\substack{s \to d^{*} \\ s \in [0, d^{*})}} g(s) = +\infty.
\end{equation}
Since \(d^{*}\) equals \(\diam (X, \rho)\), there are sequences \((x_n)_{n \in \mathbb{N}} \subseteq X\) and \((y_n)_{n \in \mathbb{N}} \subseteq X\) such that
\begin{equation}\label{t2.25:e8}
\lim_{n \to \infty} \rho(x_n, y_n) = d^{*}.
\end{equation}
In addition, by Theorem~\ref{t2.24}, we have \(\rho(x, y) < d^{*}\) for all \(x\), \(y \in X\). Consequently, the inclusion \(D(X, \rho) \subseteq [0, d^{*})\) holds. Now Lemma~\ref{l2.6} implies that the mapping \(d \colon X \times X \to \RR^{+}\) satisfying the equality
\[
d(x, y) = g(\rho(x, y))
\]
for all \(x\), \(y \in X\) is an ultrametric on~\(X\). From the second equality in \eqref{t2.25:e5} and equality~\eqref{t2.25:e8} it follows that \((X, d)\) is unbounded. A direct calculation shows the equalities
\begin{equation}\label{t2.25:e6}
f(g(s)) = s \quad \text{and} \quad g(f(t)) = t
\end{equation}
hold for all \(s \in [0, d^{*})\) and \(t \in [0, +\infty)\), where \(f\) is defined by~\eqref{t2.25:e2}. Now equality \eqref{t2.25:e1} follows from \eqref{t2.25:e6}.
\end{proof}

\begin{remark}\label{r2.33}
The condition \(|X| \geqslant 2\) cannot be dropped in the second part of Theorem~\ref{t2.25}. Indeed, if \(|X| = 1\), then, for every metric \(\rho\), the metric space \((X, \rho)\) is bounded and ultrametric with empty diametrical graph \(G_{X, \rho}\) and there are no ultrametrics \(d \colon X \times X \to \RR^{+}\) for which \(\diam (X, d) = +\infty\) holds.
\end{remark}

\begin{lemma}\label{c5.36}
Let \((X, d)\) and \((Y, \rho)\) be nonempty weakly similar ultrametric spaces. Then the diametrical graph \(G_{X, d}\) is empty if and only if the diametrical graph \(G_{Y, \rho}\) is empty.
\end{lemma}

\begin{proof}
Let \(\Phi \colon X \to Y\) be a weak similarity of \((X, d)\) and \((Y, \rho)\) with the scaling function \(f \colon D(Y) \to D(X)\). Since \(f\) is bijective and strictly increasing, the set \(D(X)\) has the largest element iff \(D(Y)\) contains the largest element.
To complete the proof it suffices to remember that the largest element of the distance set of metric space, if such an element exists, coincides with the diameter of the space.
\end{proof}

Using the concept of weak similarity we can give a more compact variant of Theorem~\ref{t2.25}.

\begin{theorem}\label{t2.35}
Let \((X, d)\) be an ultrametric space with \(|X| \geqslant 2\). Then the following statements are equivalent:
\begin{enumerate}
\item\label{t2.35:s1} \((X, d)\) is weakly similar to an unbounded ultrametric space.
\item\label{t2.35:s2} The diametrical graph \(G_{X, d}\) is empty.
\end{enumerate}
\end{theorem}

\begin{proof}
\(\ref{t2.35:s1} \Rightarrow \ref{t2.35:s2}\). Let \((X, d)\) be a weakly similar to an unbounded ultrametric space \((Y, \rho)\). Then the diametrical graph \(G_{Y, \rho}\) is empty. Hence, \(G_{X, d}\) is also empty by Lemma~\ref{c5.36}.

\(\ref{t2.35:s2} \Rightarrow \ref{t2.35:s1}\). Suppose that the diametrical graph \(G_{X, d}\) is empty. If \((X, d)\) is unbounded, then \(\ref{t2.35:s2}\) is valid because \((X, d)\) is weakly similar to itself. If \((X, d)\) is bounded, then, by Theorem~\ref{t2.25}, there is an unbounded ultrametric space \((Y, \rho)\) such that
\[
d(x, y) = \diam X \frac{\rho(x, y)}{1 + \rho(x, y)}
\]
for all \(x\), \(y \in X\). It was shown in the proof of Theorem~\ref{t2.25} that \((X, d)\) and \((Y, \rho)\) are weakly similar.
\end{proof}

%\bibliographystyle{plain}
%\bibliography{bib2021.03}

\end{document}